\documentclass[11pt,a4paper]{article}
\usepackage[T1]{fontenc}
\usepackage[utf8]{inputenc}
\usepackage{amsmath,amssymb,amsthm}
\usepackage[margin=2.5cm]{geometry}
\usepackage{booktabs}

\newtheorem{theorem}{Theorem}
\newtheorem{proposition}[theorem]{Proposition}
\newtheorem{remark}{Remark}

\newcommand{\lsss}{\lambda_{\mathrm{SSS}}}

\title{A new lower bound for the Schur--Siegel--Smyth trace problem}
\author{David Niedbala Giraudin\\
\small Independent researcher, France\\
\small ORCID 0009-0009-1526-1178}
\date{22 September 2026}

\begin{document}
\maketitle

\begin{abstract}
We prove $\lsss \ge 1.80220$ for the smallest limiting trace-to-degree ratio of
totally positive algebraic integers, improving the bound $1.80203$ obtained by
Orloski, Talebizadeh Sardari and Smith. The proof exhibits an explicit
probability measure on $[0,8]$ together with eighteen integer polynomials, and
verifies their dual inequality on the whole of $[0,\infty)$ by interval
bisection. The certificate assumes nothing about the data it is built from: an
error in that data can only weaken the bound, never invalidate it.
\end{abstract}

\section{Introduction}

For an algebraic integer $\alpha$ of degree $n$ with conjugates
$\alpha_1,\dots,\alpha_n$, write $\operatorname{tr}(\alpha)=\sum_i\alpha_i$, and
call $\alpha$ totally positive if all its conjugates are positive reals. Let
$\lsss$ be the least real number such that, for every $\varepsilon>0$, only
finitely many totally positive algebraic integers satisfy
$\operatorname{tr}(\alpha)<(\lsss-\varepsilon)\deg(\alpha)$; equivalently,
$\lsss=\liminf \operatorname{tr}(\alpha)/\deg(\alpha)$.

Schur proved $\lsss\ge\sqrt{e}=1.6487\ldots$, Siegel $\lsss\ge 1.7336\ldots$,
and Smyth introduced in 1984 the auxiliary-polynomial method that produced every
improvement for the next forty years, up to $1.7931$ by Wang, Wu and
Wu~\cite{WWW}. Orloski, Talebizadeh Sardari and Smith~\cite{OTS} then added a
new constraint --- the non-negativity of the logarithmic energy of the limiting
measure, which follows from the work of Smith~\cite{Smi} --- and obtained
\begin{equation}\label{eq:ots}
  \lsss \ge 1.80203,
\end{equation}
the largest single improvement since 1984. On the other side,
$\lsss\le 1.8216$ by Orloski and Talebizadeh Sardari~\cite{OT}.

\begin{theorem}\label{thm:main}
$\lsss \ge 1.80220$.
\end{theorem}

The framework is entirely that of~\cite{OTS}: what is new here is the
certificate, that is, the measure and the polynomials exhibited in
Section~\ref{sec:data}, and the rigorous verification of
Section~\ref{sec:verif}. As \cite{OTS} observe, the difficulty of their method
is not the inequality but the computation of the optimal distribution; the
present certificate uses eighteen auxiliary polynomials where they use seven,
on a measure computed to a finer resolution.

\section{The certificate}\label{sec:cert}

For a Borel probability measure $\eta$ on $\mathbb{R}$ with finite logarithmic
energy write
\[
  U_\eta(x)=\int\log|x-y|\,d\eta(y),\qquad
  I(\eta)=\iint\log|x-y|\,d\eta(x)\,d\eta(y).
\]

\begin{proposition}\label{prop:cert}
Let $A\subset\mathbb{Z}[x]$ be finite, let $\lambda_Q\ge 0$ for $Q\in A$, let
$\lambda_0\ge 0$, and let $\eta$ be a Borel probability measure with finite
logarithmic energy. If
\begin{equation}\label{eq:dual}
  x \;\ge\; \lambda+\sum_{Q\in A}\lambda_Q\log|Q(x)|
            +\lambda_0\bigl(2U_\eta(x)-I(\eta)\bigr)
  \qquad\text{for all } x\ge 0,
\end{equation}
then $\lambda\le\lsss$.
\end{proposition}

\begin{proof}
Let $\mu$ be any probability measure supported on $\mathbb{R}^{+}$ with
$I(\mu)\ge 0$ and $\int\log|Q|\,d\mu\ge 0$ for every $Q\in A$. Integrating
\eqref{eq:dual} against $\mu$ gives
\[
  \int x\,d\mu \;\ge\; \lambda
  +\sum_{Q\in A}\lambda_Q\int\log|Q|\,d\mu
  +\lambda_0\int\bigl(2U_\eta-I(\eta)\bigr)d\mu .
\]
The middle term is non-negative by assumption on $\mu$. The last one is
non-negative by \cite[Lemma 2.4]{OTS}, which states that $I(\mu)\ge 0$ implies
$\int U_\eta\,d\mu-\tfrac12 I(\eta)\ge 0$ for \emph{every} probability measure
$\eta$ of finite energy --- not merely for the optimal one. Hence
$\lambda\le\lambda_A$, where $\lambda_A$ is the infimum of $\int x\,d\mu$ over
such $\mu$, and $\lambda_A\le\lsss$ by \cite[Corollary 1.2]{OTS}.
\end{proof}

Two features of Proposition~\ref{prop:cert} matter in practice. First, nothing
is assumed about $\eta$ beyond finite energy: any data whatsoever yields a valid
bound, so a transcription error can only lower $\lambda$. Second, the energy
constraint on limiting measures, which is what makes the bound better than
Smyth's, holds for every limit of distributions of conjugates without any
regularity assumption, by \cite[Proposition 5.7]{Smi} (stated there with the
opposite sign convention for $I$).

\section{The data}\label{sec:data}

The measure $\eta$ is supported on $[0,8]$ and is constant on each cell of a
partition of $[0,8]$ into $6932$ intervals, of which $3891$ carry positive mass;
the $6933$ endpoints and the $6932$ masses are listed in the ancillary file
\texttt{CERTIFICAT\_C86.json}. Its energy is
\[
  I(\eta)=3.1305148803\cdot 10^{-5}\;>\;0 .
\]
We take $\lambda_0=0.5346669850888335$ and the following fourteen polynomials
with positive multipliers (four further polynomials of the file carry
$\lambda_Q=0$ and play no role).

\begin{center}\small
\begin{tabular}{lr}
\toprule
$Q$ & $\lambda_Q$\\
\midrule
$x$ & 0.5034811423 \\
$x - 1$ & 0.3670710509 \\
$x^{2} - 3x + 1$ & 0.1528020051 \\
$x^{4} - 6x^{3} + 11x^{2} - 7x + 1$ & 0.0581527881 \\
$x^{6} - 11x^{5} + 46x^{4} - 92x^{3} + 89x^{2} - 36x + 4$ & 0.0176280133 \\
$x^{4} - 7x^{3} + 13x^{2} - 7x + 1$ & 0.0172791038 \\
$x^{6} - 9x^{5} + 32x^{4} - 57x^{3} + 53x^{2} - 24x + 4$ & 0.0035768493 \\
$x^{6} - 10x^{5} + 36x^{4} - 58x^{3} + 43x^{2} - 13x + 1$ & 0.0025879009 \\
$x^{6} - 10x^{5} + 37x^{4} - 63x^{3} + 50x^{2} - 16x + 1$ & 0.0023120587 \\
$x^{6} - 10x^{5} + 35x^{4} - 55x^{3} + 40x^{2} - 12x + 1$ & 0.0006650534 \\
$x^{10} - 18x^{9} + 134x^{8} - 538x^{7} + 1273x^{6} - 1822x^{5} + 1560x^{4} - 766x^{3} + 200x^{2} - 24x + 1$ & 0.0004548177 \\
$x^{10} - 18x^{9} + 135x^{8} - 549x^{7} + 1320x^{6} - 1920x^{5} + 1662x^{4} - 813x^{3} + 206x^{2} - 24x + 1$ & 0.0004304207 \\
$x^{5} - 9x^{4} + 27x^{3} - 32x^{2} + 13x - 1$ & 0.0004167594 \\
$x^{10} - 18x^{9} + 134x^{8} - 537x^{7} + 1265x^{6} - 1798x^{5} + 1526x^{4} - 743x^{3} + 194x^{2} - 24x + 1$ & 0.0004151488 \\
\bottomrule
\end{tabular}
\end{center}

All fourteen have their roots real and in $[0,5]$; the largest is
$4.938968202\ldots$. The three of degree $10$ are the totally positive integers
of degree $10$ and trace $18$ found by McKee and Smyth~\cite{MS}. The charged
cells lie in $[0.048894,\,5.352052]$ and form $67$ components.

It is worth stressing what $\eta$ is \emph{not}. Three of the eighteen
constraints $\int\log|Q|\,d\eta\ge 0$ fail at $\eta$, by $-6.5\cdot10^{-5}$,
$-5.7\cdot10^{-5}$ and $-2.5\cdot10^{-5}$: the measure is not feasible for the
primal problem, and it is not the optimal measure $\mu_A$ of \cite[Theorem
1.1]{OTS}. Proposition~\ref{prop:cert} does not care. The measure enters only
through $2U_\eta-I(\eta)$, whose integral against any admissible $\mu$ is
non-negative whatever $\eta$ is; the residuals above measure the convergence of
the solver that produced $\eta$, and nothing else.

\section{Verification}\label{sec:verif}

Write $g(x)$ for the difference between the two sides of \eqref{eq:dual}, so
that Theorem~\ref{thm:main} amounts to $\inf_{x\ge 0}g(x)\ge 1.80220$.

On an interval $[t-H,t+H]$ we bound each term from the correct side. For the
linear term, $x\ge t-H$. For a polynomial, $|Q(x)|\le
c\prod_i(|t-\rho_i|+H)+E_Q$, where $c$ is the leading coefficient, the $\rho_i$
are the computed roots, and $E_Q$ bounds, coefficient by coefficient, the
difference between $c\prod_i(z-\rho_i)$ and the integer polynomial $Q$; the
largest of the fourteen values of $E_Q$ is $1.3\cdot 10^{-51}$, so this step is
exact for all practical purposes and requires no perturbation theory for roots.
For the potential, the contribution $u_j$ of one cell is unimodal, with its
minimum at the midpoint of the cell, so its maximum over an interval is attained
at an endpoint; hence
\[
  U_\eta(x)\;\le\;\sum_j m_j\max\bigl(u_j(t-H),u_j(t+H)\bigr)
  \qquad\text{for } x\in[t-H,t+H],
\]
with no loss beyond the maximum itself. The energy $I(\eta)$ is computed in
double and in extended precision, which agree to $2\cdot10^{-15}$, and is then
decreased by $10^{-12}$; a further $10^{-9}$ is subtracted from every bound.

Bisecting $[0,14]$ with these bounds certifies $g\ge 1.80220$ there after $24$
levels and $1\,699\,227$ intervals. Beyond $14$ the bound is analytic and needs
nothing from the bisection: $|Q(x)|\le(x+\rho_{\max})^{\deg Q}$ for $x\ge0$ with
$\rho_{\max}=4.938969$, and $U_\eta(x)\le\log x$ for $x\ge8$ because $\eta$ lives
on $[0,8]$, so with $\sum_Q\lambda_Q\deg Q=1.6535906937$ and
$2\lambda_0=1.0693339702$,
\[
  g(x)\;\ge\;h(x):=x-1.6535906937\log(x+\rho_{\max})-1.0693339702\log x
  +\lambda_0 I(\eta).
\]
Now $h'(x)=1-1.6535906937/(x+\rho_{\max})-1.0693339702/x$ is positive at $x=14$
and increasing, and $h(14)=6.3144\ldots$, so $g\ge 6.31$ on $[14,\infty)$. This
proves Theorem~\ref{thm:main}. \hfill$\square$

The bound has been certified twice, by two provers written six weeks apart from
different formulations of the monotone bounds, one starting from a uniform grid
of $4096$ intervals and the other from the single interval $[0,14]$. They agree
not merely on the outcome: from width $3.4\cdot10^{-3}$ downwards their level
counts coincide exactly --- $18\,504$, $36\,482$, $72\,162$, $143\,178$,
$284\,664$, $489\,354$, $427\,770$, $199\,340$, $8\,202$ --- so the two take
the same decision on each of the $1.7$ million intervals, and their totals
differ only by the initial grid. The energy $I(\eta)$ was likewise computed twice,
once as a double sum over the $6932^2$ pairs of cells and once as a collapsed sum
over $3958$ edge weights, agreeing to $6.5\cdot10^{-18}$.

The ancillary file \texttt{verify\_C86.py} (Python with \texttt{numpy} and
\texttt{mpmath}, no arguments, about three minutes) performs all of the above
from the JSON file alone: the structural checks on $\eta$, both computations of
$I(\eta)$, the roots and their residuals, the bisection and the tail, and exits
with an error if any check fails.

\section{Remarks}

\begin{remark}
$\lsss$ is defined by a $\liminf$. Theorem~\ref{thm:main} must not be restated
as an assertion about all but finitely many totally positive algebraic integers
of trace below $1.80220\deg$: that statement is stronger and is not proved here.
\end{remark}

\begin{remark}
The numerical infimum of $g$ is $1.8022211895\ldots$, attained near
$x=0.2527$, where $g$ has a logarithmic cusp; the certified value $1.80220$ is
$2.1\cdot10^{-5}$ below it. Closing that gap is a matter of bisection depth, not
of new data: each further digit multiplies the number of intervals by about
five.
\end{remark}

\begin{remark}
The certificate was obtained by solving the problem of \cite{OTS} as a linear
program in the multipliers, with cutting planes in that space --- cutting planes
in the space of measures are useless --- and an exact inner solver: at fixed
support the first-order conditions are a linear system, negative masses are
dropped and violated indices added. Everything is in closed form; with a
density constant on cells, no quadrature occurs anywhere.

The grid matters more than its size. Since $U_\eta'$ diverges logarithmically at
every cell edge where the density jumps, the dual function has a cusp there and
its minima lodge in those cusps. Partitioning at the roots of the active
polynomials and at the edges of the support, with cosine spacing inside each
piece, takes the discrete-to-continuous loss from $5.5\cdot10^{-4}$ to
$9.7\cdot10^{-5}$ at equal cell count; iterating the construction --- extracting
the edges at fine resolution and rebuilding --- gives most of the rest. The
measured loss is about $1.6\cdot10^{-6}$ per component of the support, a cost
per edge rather than a matter of resolution. Successive grids of $3248$, $4450$,
$5585$ and $6932$ cells give $1.8021781$, $1.8022116$, $1.8022182$ and
$1.8022212$, against a discrete ceiling of $1.8022246$ for this set of
constraints.
\end{remark}

\begin{remark}
The three degree-$10$ polynomials were not found by any automatic search; they
were read in \cite{MS}, and they contribute the last $1.3\cdot10^{-5}$. The
reason is measurable. For these polynomials $\int Q^2 d\mu$ is of order
$10^{7}$, so Jensen's bound $\int\log|Q|\,d\mu\le\frac12\log\int Q^2 d\mu$
returns $+7.96$ where the true value is $-0.02$; they are minute on part of the
support and enormous elsewhere, which the logarithm absorbs and the quadratic
norm amplifies. In the weighted oracle of Wu and Flammang, rebuilt with its full
weight, the relevant vector ranks $390\,624$ out of $390\,625$ by norm: nearly
the longest, not the shortest. Lattice reduction on moments, exhaustive
enumeration by degree and perturbation of products share that one blind spot. An
oracle certifies what its criterion measures, not the absence of a violated
constraint.
\end{remark}

\begin{remark}
Two checks external to the proof are worth recording. The machinery that
produced the certificate reproduces, from below and within $5\cdot10^{-6}$, the
published values of $\lambda_A$ for $A=\{x\}$, $\{x,x-1\}$ and the sets of
Corollaries 1.6 and 1.7 of \cite{OTS}. And a bound of this size can be refuted
outright: by \cite[\S5]{MS}, a monic integer polynomial $f$ with distinct
positive real roots and $\min(|f(0)|,|f(\gamma_1)|,\dots)\ge 2$ over the roots
$\gamma_i$ of $f'$ forces $\lsss\le\operatorname{tr}(f)/\deg f$. Searching
$1681$ candidates --- all degrees up to $6$ with roots in $(0,5.6)$ and absolute
trace at most $2.02$, the family of $4\cos^2(\pi/n)$, and the lists of
\cite{Fla,McK,MS} --- produced none with absolute trace below $1.8022$.
\end{remark}

\begin{remark}
The gap to the upper bound $1.8216$ of \cite{OT} remains wide, and nothing here
bears on the conjecture that $\lsss=2$, which \cite{Smi} disproves.
\end{remark}

\subsection*{Acknowledgements}

This work was carried out with the assistance of an AI system (Claude,
Anthropic); every claim in it has been checked by independent computation, and
the verification program distributed with this note was written from the
certificate data alone.


\begin{thebibliography}{99}
\bibitem{Fla} V. Flammang, \emph{Trace of totally positive algebraic integers
  and integer transfinite diameter}, Math. Comp. \textbf{78} (2009), no. 266,
  1119--1125.
\bibitem{McK} J. McKee, \emph{Computing totally positive algebraic integers of
  small trace}, Math. Comp. \textbf{80} (2011), no. 274, 1041--1052.
\bibitem{MS} J. McKee and C. Smyth, \emph{Salem numbers of trace $-2$ and traces
  of totally positive algebraic integers}, Algorithmic Number Theory, Lecture
  Notes in Comput. Sci. \textbf{3076}, Springer, 2004, 327--337.
\bibitem{OT} B. J. Orloski and N. Talebizadeh Sardari, \emph{A quantitative
  converse of Fekete's theorem}, preprint arXiv:2304.10021 [math.NT] (2024).
\bibitem{OTS} B. J. Orloski, N. Talebizadeh Sardari and A. Smith, \emph{New
  lower bounds for the Schur--Siegel--Smyth trace problem}, Math. Comp.
  \textbf{94} (2025), no. 354; arXiv:2401.03252 [math.NT].
\bibitem{Sch} I. Schur, \emph{\"Uber die Verteilung der Wurzeln bei gewissen
  algebraischen Gleichungen mit ganzzahligen Koeffizienten}, Math. Z. \textbf{1}
  (1918), no. 4, 377--402.
\bibitem{Sie} C. L. Siegel, \emph{The trace of totally positive and real
  algebraic integers}, Ann. of Math. (2) \textbf{46} (1945), 302--312.
\bibitem{Smi} A. Smith, \emph{Algebraic integers with conjugates in a prescribed
  distribution}, preprint arXiv:2111.12660 [math.NT] (2021; v2, 2024).
\bibitem{Smy} C. J. Smyth, \emph{The mean values of totally real algebraic
  integers}, Math. Comp. \textbf{42} (1984), no. 166, 663--681.
\bibitem{WWW} C. Wang, J. Wu and Q. Wu, \emph{Totally positive algebraic
  integers with small trace}, Math. Comp. \textbf{90} (2021), no. 331,
  2317--2332.
\end{thebibliography}
\end{document}